\documentclass[12pt]{article}
\usepackage{url}
\usepackage{fancyhdr}
\usepackage{setspace}  
\usepackage{mathptmx}
\usepackage{latexsym}
\usepackage{amssymb}
\usepackage{amsthm}
\usepackage{amsmath}
\newtheorem{theorem}{Theorem}

\newtheorem{corollary}{Corollary}

\begin{document}
\title{\bf Realizing the Chromatic Numbers \\and Orders of Spinal Quadrangulations \\of Surfaces}

\author{\Large Serge Lawrencenko \\Department of Mathematics and Science Education \\ 
Faculty of Service, Russian State University of Tourism and Service\\
Glavnaya 99, Cherkizovo, Moscow Region 141221, Russia \\
{\tt lawrencenko@hotmail.com}}
\maketitle

\begin{abstract}
A method is suggested for construction of quadrangulations of the closed orientable surface with given genus $g$ and either (1) with given chromatic number or (2) with given order allowed by the genus $g$. In particular, N.~Hartsfield and G. Ringel's results [J. Comb. Theory, Ser. B 46 (1989), 84-95] are generalized by way of generating minimal quadrangulations of infinitely many other genera. 
\end{abstract}



\section{Motivation}

The purpose of this note is to suggest a method for generating quadrangulations of closed orientable surfaces with given properties. Those surfaces arise as the thickenings of spatial graphs called the spines. The operation of thickening a given graph $G$ builds a surface around $G$ which is used in computer animation, as shown in \cite {bbb}. The first step of thickening $G$ consists in building a cylinder around each edge of $G$ by setting four quadrilaterals around that edge. Therefore a thickening of a spine graph is represented by its so-called spinal quadrangulation. The properties of the spinal quadrangulation are determined by its spine. In particular, one can control the genus, chromatic number, and order of the quadrangulation by choosing a suitable spine. The question arises of how wide the spectra of these three parameters can be.

\section{Spinal Quadrangulations}

If a (finite, undirected, simple) graph $G$ is 2-cell embedded in the sphere with $g$ handles, $\Sigma_g$, the components of $\Sigma_g-G$ are called {\it regions}. A {\it quadrangulation of $\Sigma_g$ with graph $G$} is a 2-cell embedding $G \hookrightarrow \Sigma_g $ in which each region is bounded by a simple circuit of length $4$ in $G$. The {\it chromatic number} $\chi (G)$ of $G$, as well as of any  quadrangulation with this graph, is the minimum number of colors sufficient for coloring the vertices of $G$ so that adjacent vertices receive different colors. The {\it first Betti number} of a connected graph $G$ is given by $\beta (G)=|E(G)|-|V(G)|+1$, where $|V(G)|$ and $|E(G)|$ stand for the cardinalities of the vertex and edge sets of $G$ (respectively). The cardinality $|V(G)|$ is also called the {\it order} of $G$. Especially, for the complete graph $K_n$ of order $n$, $\beta(K_n)=\frac 1 2 (n-1)(n-2)$. 

The {\it 2-fold interlacement} of $G$ is denoted by ${G [:]}$ and is defined to be the graph which has vertex set ${V(G[:])}=V(G') \sqcup V(G'')$, where  $G'$ and $G''$ are two disjoint copies of $G$, and as edges the set ${E(G[:])}=E(G') \sqcup E(G'')$ plus the edges that join each vertex $v' \in V(G')$ to each vertex in $V(G'')$ that is adjacent (in $G''$) to the corresponding vertex $v'' \in V(G'')$. For instance, $K_n[:] = K_{n(2)}$---the general octahedral graph $O_n$ which is the complement of a 1-factor in the complete graph $K_{2n}$ ($n\ge 2$).

\begin{theorem} [\bf White \cite {w}, Craft \cite {c}]
For any non-trivial connected graph $G$, there exists a quadrangulation ${G[:]} \hookrightarrow \Sigma_{\beta(G)}$.
\end{theorem}

Any quadrangulation ${G[:]} \hookrightarrow \Sigma_{\beta(G)}$ is called a {\it spinal quadrangulation} with {\it spine} $G$ and {\it genus} $g=\beta(G)$. By Theorem 1, the genus of a spinal quadrangulation is equal to the first Betti number of the spine. The following corollary is a quadrilateral analogue of the result of \cite {hlk} on triangulations.

\begin{corollary}
For any integers $g \ge 0$ and $n \ge 2$ such that $g \ge \beta(K_n)=\frac 1 2 (n-1)(n-2)$, there exists a spinal quadrangulation of $\Sigma_g$ with chromatic number $n$.
\end{corollary}

\begin{proof}
Since it is obvious that ${\chi (G[:])} = \chi (G)$, it follows from Theorem 1 that we can use as a spine any graph $G$ with $\beta (G) = g$ and $\chi (G) = n$. To construct such a spine, we start up with $K_n$. If $\beta (K_n) = g$, we're done. If $\beta (K_n) < g$, we take the ladder graph $L_{g-\beta(K_n)+1}$ with $g-\beta(K_n)$ linearly independent cycles (of length 4) and glue one of its ends on $K_n$. The {\it ladder graph} is defined to be the Cartesian product of two path graphs, one of which has only one edge: $L_n = P_n \times P_1$.
\end{proof}

The spectrum $\{n\}$ described in Corollary 1 for possible chromatic numbers is full for any fixed genus $g \ge 0$. In fact, $\frac 1 2 (\chi(G)-1)(\chi(G)-2) = \frac 1 8 (2\chi(G)-3)^2-\frac 1 8 \le \beta(G)$, where the inequality is provided by \cite{chs, cz}, whence the maximum chromatic number $n$ of a graph with given first Betti number $g \ge 0$ is attained by the largest complete graph $K_n$ such that $\beta(K_n) \le g$.

Yet, the question remains of what the limitations of the proposed approach are when speaking about the whole spectrum of possible chromatic numbers for a quadrangulation of a given surface. If $K_n$ has a quadrilateral embedding in $\Sigma_g$ (the case $n\equiv 5 \mod{8}$ was shown in \cite{hr}), then $g=1+ \frac 1 8 n(n-5)$ and one may expect that for each $h$ such that $2 \le h \le n$ there would be a quadrangulation of $\Sigma_g$ with chromatic number $h$. Meanwhile, the largest integer $m$ such that
$$
\beta(K_m) = \frac12 (m-1)(m-2) \le g=1+\frac18 n(n-5)
$$
is $m=\lfloor \frac12 (3+\sqrt{8g+1})\rfloor \approx \frac n 2$ for large $n$. Therefore, by using quadrilateral embeddings of spine graphs, we manage to cover about a half of the possible values of the chromatic number.

\section{Minimal Quadrangulations}

We start up  with two corollaries of Theorem 1.

\begin{corollary}
For any integer $p \ge 2$ there exists a quadrangulation ${K_p[:]} \hookrightarrow \Sigma_g$ with $g = \beta(K_p)=\frac 1 2 (p-1)(p-2)$.
\end{corollary}

\begin{corollary}
Let $p$ be an integer $\ge3$ and let $(K_p-e)$ be any graph formed by deleting an edge from $K_p$. Then there exists a quadrangulation ${(K_p-e)[:]} \hookrightarrow \Sigma_g$ with $g=\beta(K_p)-1$.
\end{corollary}

A quadrangulation of a fixed surface $\Sigma_g$ is said to be {\it minimal} in $\Sigma_g$ if the number of vertices is minimal among all (not necessarily spinal) quadrangulations of $\Sigma_g$. The quadrangulations of Corollaries 2 and 3 were first discovered (in terms of general octahedral graphs) by Hartsfield and Ringel \cite{hr}, who also showed their minimality in $\Sigma_g$ whenever $p \ge 4$ for Corollary 2, and $p \ge 8$ for Corollary 3. These corollaries are special cases of the main theorem of this note, stated shortly, and correspond to the particular cases $m=0$ and $m=1$, respectively. Using the method of current graphs for $m \ge 2$ would have been very complicated, so Hartsfield and Ringel had to stop at $m=1$. In contrast, the spinal method enables generalization of the results of \cite{hr} to an arbitrary $m$ not exceeding $\frac 14 p-1$. This demonstrates one way in which the spinal method is useful.

\begin{theorem}
Let $(K_p-me)$ be any graph formed by deleting $m$ edges from $K_p$ ($0 \le m \le \beta(K_p)$). If $(K_p-me)$ is connected, there exists a quadrangulation ${(K_p-me) [:]} \hookrightarrow \Sigma_g$ with $g = \beta(K_p)-m$. Moreover, any such quadrangulation is minimal in $\Sigma_g$ with $g \ge 1$ whenever $p \ge 4(m+1)$.
\end{theorem}

\begin{proof}
Compute $g=\beta(K_p-me) = \frac12 (p-1)(p-2)-m$, and the existence follows from Theorem 1. It remains to prove minimality of the constructed quadrangulation. Let $\alpha_0$, $\alpha_1$, $\alpha_2$ denote the number of vertices, edges, and regions of an arbitrary quadrangulation of $\Sigma_g$. By Euler's equation, we have $\alpha_0 - \alpha_1 + \alpha_2 = 2-2g$. Furthermore, since any pair of vertices are joined by at most one edge, we have $\alpha_1 \le {\alpha_0\choose 2}$, and since each edge meets exactly two regions, we have $4\alpha_2 = 2\alpha_1$. From these it can be derived that $\alpha_0^2 –-5\alpha_0 +(8-8g) \ge 0$. This quadratic inequality has the solution:
\begin{equation}
\alpha_0 \ge \Big\lceil \frac 12 \Big{(}5+\sqrt{32g-7} \Big{)} \Big\rceil.
\end{equation}
With $g$ computed in the beginning of the proof, we find
\begin{equation}
32g-7 = 16p^2 -48p+25-32m. 
\end{equation}
Now, since the constructed quadrangulation has $2p$ vertices, it follows from Eqs. (1) and (2) that it is minimal in $\Sigma_g$ whenever the following double inequality holds:   
$$
2p-1 < \frac12 \Big{(} 5+\sqrt{16p^2-48p+25-32m} \Big{)} \le 2p.
$$
This can be rewritten as $16p^2-56p+49<16p^2-48p+25-32m\le16p^2-40p+25$, or as
$$
\left\{     \begin{array}{lr}       8p>24+32m,\\      8p \ge -32m.    \end{array}   \right.
$$
Since the second inequality is a tautology, we conclude that $p>3+4m$, hence $p\ge4+4m$.         
The proof is complete. 
\end{proof}

Theorem 2 provides infinitely many new minimal quadrangulations (for infinitely many genera $g$) not covered by Hartsfield and Ringel \cite{hr}. The new minimal quadrangulations correspond to the values of $m$ satisfying the double inequality: $2 \le m \le {\frac 14 p -1}$. For example, the quadrangulations ${K_{12}[:]} \hookrightarrow \Sigma_{55}$, ${(K_{12}-e)[:]} \hookrightarrow \Sigma_{54}$, and ${(K_{12}-2e)[:]} \hookrightarrow \Sigma_{53}$ are minimal for the corresponding surfaces (and have 132, 130, and 128 regions, respectively). The first two are covered by \cite{hr} (or Corollaries 2 and 3), but the minimal quadrangulation on $\Sigma_{53}$ is a new one. Note that a quadrangulation is minimal in the sense of \cite{hr} if it has the minimum number of regions, but since the definition assumes the surface is fixed, their definition agrees with the one given in the beginning of this section.

The construction using spines creates quadrangulations of an easily controlled order. The following is an analogue of Corollary 1.

\begin{corollary}
For any integers $g \ge 0$ and $p \ge 2$ such that $g \le \beta(K_p)$, there exists a spinal quadrangulation of $\Sigma_g$ with order $2p$.
\end{corollary}
\begin{proof}
The order of a spinal quadrangulation is twice the order of its spine. By Theorem 2, we can take $(K_p-me)$ as a spine, letting $m$ be $(\beta(K_p)-g)$. Clearly, it is possible to remove this number of edges from $K_p$ so that the remaining graph is still connected.
\end{proof}

The spectrum $\{2p\}$ described in Corollary 4 for the possible orders is full for fixed genus; there are infinitely many spinal quadrangulations of arbitrarily large even order in each surface $\Sigma_g$. Solving the quadratic equation $\beta(K_p)=g$ for $p$, we come to the following formula for the order of a minimal spinal quadrangulation with genus $g\ge0$: 
\begin{equation}
{|V|_{\min}^{{\rm{sp}}}(g)} = 2 \Big{\lceil} \frac12 (3+\sqrt{8g+1}) \Big{\rceil}
\end{equation}
and therefore
\begin{equation}
{|V|_{\min}(g)} \le 2 \Big{\lceil} \frac12 (3+\sqrt{8g+1}) \Big{\rceil},
\end{equation}
where ${|V|_{\min}(g)}$ denotes the order of a general minimal quadrangulation with genus $g$. It is surprising that no comprehensive formula is available for ${|V|_{\min}(g)}$.
We'll shortly provide a partial formula for ${|V|_{\min}(g)}$ and show that ${|V|_{\min}(g)}$ reaches the upper bound (4) infinitely often.

For example, ${|V|_{\min}^{{\rm{sp}}}(0)} = {|V|_{\min}(0)} =4$ by Eq. (3). It should be noted that Hartsfield and Ringel \cite{hr} assert that ${|V|_{\min}(0)} =8$ because their definition of a quadrangulation, in comparison to the definition given in the beginning of Section 2, has an additional requirement as follows: the intersection of any two distinct regions is either empty or at most one edge and at most three vertices. However, any minimal quadrangulation in the sense of our definition has no vertices of degree 2 whenever $g\ge1$, and therefore satisfies this requirement. By Eq. (3), ${|V|_{\min}^{{\rm{sp}}}(1)} =6$ and ${|V|_{\min}^{{\rm{sp}}}(2)} =8$, whereas it is shown in \cite{hr} that ${|V|_{\min}(1)} =5$ and ${|V|_{\min}(2)} =7$.

\begin{corollary}
For $g \ge 3$, ${|V|_{\min}(g)} = 2 \lceil {a(g)} \rceil$ whenever $\lceil a(g) \rceil = \lfloor b(g) \rfloor$,
\noindent where $a(g)=\frac12 (3+\sqrt{8g+1})$ and $b(g)=\frac14 (7+\sqrt{32g-15})$.  
\end{corollary}
\begin{proof}
It follows from Theorem 2 that ${|V|_{\min}(g)} = {|V|_{\min}^{{\rm{sp}}}(g)}=2p$ whenever this double inequality holds: $\beta(K_p)+1- \frac{p}4 \le g \le \beta(K_p)$ or, equivalently, $a(g) \le p \le b(g)$.
Note that whenever $g\ge3$ (and so $p \ge 4$) we have $a(g) \le b(g)$, and also observe that the closed interval $[a(g), b(g)]$ has length $<1$ and can contain at most one integer, $\ell(g)$. Furthermore, such an $\ell(g)$ exists if $\lceil a(g) \rceil = \lfloor b(g)\rfloor$ ($=\ell(g)$), in which case we have ${|V|_{\min}(g)} = {|V|_{\min}^{{\rm{sp}}}(g)} =2\ell(g)$. 
\end{proof}

The formula of Corollary 5 applies to infinitely many values of $g$. For example, for each $p\ge4$ by letting $g = \beta(K_p)$ we find $a(g)=p$ ($=\ell(g)$), and so ${|V|_{\min}(g)}=2p$.
\medskip

\bibliographystyle{model1-num-names}
\bibliography{<your-bib-database>}

\end{document}